\documentclass{amsart}
\usepackage{amsmath, amssymb, mathrsfs, verbatim, multirow}
\usepackage[all]{xy}
\usepackage{pifont}
\usepackage{float}

\newtheorem{Teo}{Theorem}[section]
\newtheorem{Prop}[Teo]{Proposition}
\newtheorem{Lema}[Teo]{Lemma}
\newtheorem{Cor}[Teo]{Corollary}

\theoremstyle{definition}
\newtheorem{Def}[Teo]{Definition}

\newtheorem{Obs}[Teo]{Remark}
\newtheorem{Que}[Teo]{Question}
\newtheorem{Exa}[Teo]{Example}

\newcommand{\Q}{\mathbb{Q}}

\newcommand{\N}{\mathbb{N}}

\newcommand{\Llr}{\Longleftrightarrow}
\newcommand{\lra}{\longrightarrow}
\newcommand{\Lra}{\Longrightarrow}

\newcommand{\SU}{\mbox{\rm supp}}

\DeclareMathOperator{\inv}{in}
\begin{document}
\title{On MacLane-Vaqui\'e key polynomials}
\author{Josnei Novacoski}
\thanks{During the realization of this project the author was supported by a grant from Funda\c c\~ao de Amparo \`a Pesquisa do Estado de S\~ao Paulo (process number 2017/17835-9).}

\begin{abstract}
One of the main goals of this paper is to present the relation of limit key polynomials and limit MacLane-Vaqui\'e key polynomials. This is a continuation of the work started in \cite{Spivmahandjul}, where it is proved the equivalent result for key polynomials (which are not limit). Moreover, we present a result (Theorem \ref{lemasobreooutrocoiso}) that generalizes various results in the literature.
\end{abstract}


\maketitle
\section{Introduction}
The concept of key polynomials was introduced in \cite{Mac_1} and \cite{Mac_2} in order to understand extensions of a valuation $\nu_0$ on a field $K$ to the field $K(x)$. The ideia is that for a given valuation $\nu$ on $K(x)$, a key polynomial $Q\in K[x]$ allows us to build new valuations $\nu'$ with $\nu\leq \nu'$ (i.e., $\nu(f)\leq \nu'(f)$ for every $f\in K[x]$), such that $\nu(Q)<\nu'(Q)$. MacLane proved that if $\nu_0$ is discrete, then every valuation $\nu$ on $K(x)$, extending $\nu_0$, can be build by starting with a \textit{monomial valuation} and use a sequence (of order type at most $\omega$) built iteractively to obtain $\nu$.

A major development was presented by Vaqui\'e in \cite{Vaq_1} and \cite{Vaq_2}. He introduced the concept of limit key polynomial, and proved that if we allow these objects in the sequence, then we can drop the assumption of $\nu_0$ being discrete in MacLane's main result (in this case, the order type of the sequence can be larger than $\omega$). Key polynomials as defined by MacLane and Vaqui\'e will be called MacLane-Vaqui\'e key polynomials in this paper.

An alternative definition of key polynomials was introduced in \cite{Spivmahandjul} and \cite{SopivNova} (in \cite{Spivmahandjul} they are called abstract key polynomials). The main difference between these two objects is that Maclane-Vaqui\'e's key polynomials allow us to extend a valuation, while key polynomials allow us to \textit{truncate} a valuation. In particular, if we start with a valuation $\nu$ on $K[x]$ and consider the valuation $\nu'$ obtained by MacLane-Vaqui\'e's method, then $\nu\leq\nu'$ and by the method in \cite{Spivmahandjul} and \cite{SopivNova}, we obtain $\nu'\leq\nu$. Because of this, key polynomials are more closely related to other similar objects in the literature, such as \textit{pseudo-convergent sequences} as defined in \cite{Kap} and \textit{minimal pairs} as defined in \cite{Kand2}. The relations between these objects were explored in \cite{Nov} and \cite{SopivNova}.

In \cite{Spivmahandjul}, the relation between key polynomials and MacLane-Vaqui\'e key polynomials have been studied. One of the main results of this paper (Theorem \ref{maintheomrlimitvsvauqiemabca}) is to extend their result to limit key polynomials. Some of the results contained here have been proved in the papers cited above. We decided to rewrite them here, because we believe our proofs are simpler. Also, this makes most of this paper self-contained.

The main strategy to build valuations with the objects mentioned above is the following. Let $\Gamma$ be an ordered abelian group and $\Gamma_\infty:=\Gamma\cup\{\infty\}$ where $\infty$ is a symbol not in $\Gamma$ and the extension of order and addition extends in the obvious way from $\Gamma$ to $\Gamma_\infty$. For a given $n\in\N$ we define
\[
K[x]_n:=\{p(x)\in K[x]\mid\deg(p)<n\}.
\]
Let $\mu:K[x]_n\lra \Gamma_\infty$ be a map and let $q\in K[x]$ be a polynomial of degree $n$, $\Gamma'$ a group containing $\Gamma$ and $\gamma\in \Gamma'$. For every polynomial $f\in K[x]$ we can write uniquely
\[
f=f_0+f_1q+\ldots+f_rq^r\mbox{ with }f_i\in K[x]_n\mbox{ for every }i, 0\leq i\leq n.
\]
The expression above is called the \textbf{$q$-expansion of $f$}. Hence, we can define a map $\mu':K[x]\lra \Gamma'_\infty$ by
\[
\mu'\left(f\right)=\min\{\mu(f_i)+i\gamma\}.
\]
Various results in the literature show that under certain assumptions on $\mu$, $q$ and $\gamma$, the map $\mu'$ is a valuation. Examples of this are Theorems 3.1 and 4.2 of \cite{Mac_1}, Proposition 1.22 of \cite{Vaq_1}, Proposition 2.6 of \cite{SopivNova} (which is the same as Proposition 15 of \cite{Spivmahandjul}) and Theorem 3 of \cite{Kap}. Our next result gives a criterium for $\mu$ in order for $\mu'$ to be a valuation. All the results mentioned above follow as corollaries of our theorem.

Let $S$ be a subset of a ring $R$ and $\Gamma$ an ordered abelian group.  We will consider the following properties for a map $\mu:S\lra \Gamma_\infty$. 
\begin{description}
\item[(V1)] For every $f,g\in S$, if $fg\in S$, then
\[
\mu(fg)=\mu(f)+\mu(g).
\]
\item[(V2)] For every $f,g\in S$, if $f+g\in S$, then
\[
\mu(f+g)\geq\min\{\mu(f),\mu(g)\}.
\]
\end{description}

\begin{Teo}\label{lemasobreooutrocoiso}
Let $S$ be a subset of $K[x]$ closed by multiplication with $K[x]_n\subseteq S$ and $q\in K[x]$ a polynomial of degree $n$. Assume that $\mu:S\lra \Gamma_\infty$ satisfies \textbf{(V2)} and that for every $\overline f,\overline g\in K[x]_n$ we have
\begin{description}
\item[(i)] $\mu(\overline f\overline g)=\mu(\overline f)+\mu(\overline g)$; and
\item[(ii)] if $\overline f\overline g=aq+c$ with $c\in K[x]_n$ (and consequently $a\in K[x]_n$), then
\[
\mu(c)=\mu(\overline f\overline g)<\mu(a)+\gamma.
\]
\end{description}
Then $\mu'$ satisfies the property \textbf{(V1)} and \textbf{(V2)}.
\end{Teo}

Recently, in \cite{leloup}, Leloup defines key polynomials essentially as elements which satisfy the property \textbf{(ii)} above. Our theorem shows that indeed, this is the property that allow us to extend or contruct valuations.

This paper is divided as follows. Section \ref{proofodthrems11} is devoted to prove Theorem \ref{lemasobreooutrocoiso} and present a first corollary (the remaining consequences of Theorem \ref{lemasobreooutrocoiso} will appear in their respective sections). In Section \ref{keypolynomials} we recall the definition of key polynomials and their main properties. We also add a few results that do not appear in \cite{Spivmahandjul} or \cite{SopivNova} and will be needed here. In Section \ref{gradedalgebras}, we present some results about the graded algebra of a truncated valuation. These results will be essential to prove the comparison results. In Section \ref{Maclanevaquie}, we present the theory of MacLane-Vaqui\'e key polynomials. Finally, in Section \ref{comparison} we prove the comparison results between key polynomials and MacLane-Vaqui\'e key polynomials.
\section{Proof of Theorem \ref{lemasobreooutrocoiso}}\label{proofodthrems11}
We start by presenting some basic definitions.
\begin{Def}
A map $\nu:R\lra\Gamma_\infty$ is called a valuation if it satisfies \textbf{(V1)}, \textbf{(V2)} and
\begin{description}
\item[(V3)] $\nu(1)=0$ and $\nu(0)=\infty$.
\end{description}
\end{Def}

The set $\SU(\nu):=\{a\in R\mid\nu(a)=\infty\}$ is a prime ideal of $R$. In particular, if $R$ is a field, then this definition is the same as the classic definition of valuation on a field. Also, if $R=K[x]$, then the restriction $\nu_0$ of a valuation $\nu$ on $R$ is a valuation on $K$. There is a bijection between valuations $\nu$ on $K[x]$ whose restriction to $K$ is $\nu_0$ and all valuations on simple extensions of $K$ extending $\nu_0$. More precisely, valuations $\nu$ for which $\SU(\nu)=(0)$ correspond to valuations on simple transcendental extensions of $(K,\nu_0)$ and valuations $\nu$ for which $\SU(\nu)\neq (0)$ correspond to valuations on simple algebraic extensions of $(K,\nu_0)$.

\begin{Obs}
Some authors use the term valuation to refer only to valuations with $\SU(\nu)=0$ and \textit{pseudo-valuations} for those which $\SU(\nu)\neq 0$. Also, some authors use the term \textit{Krull valuations} for those which $\SU(\nu)=0$. We do not make such distinctions because we do not find it necessary.
\end{Obs}
We denote by $\nu(R)$ the image of all elements of $R$ in $\Gamma$, i.e.,
\[
\nu(R)=\{\nu(a)\mid a\in R\setminus \SU(\nu)\}.
\]
We start with the following lemma.
\begin{Lema}\label{Lemasobreextsvelidesitra}
If $\mu$ satisfies \textbf{(V2)}, then $\mu'$ also satisfies \textbf{(V2)}.
\end{Lema}
\begin{proof}
Take two polynomials $f,g\in K[x]$ and write (adding zero terms if necessary)
\[
f(x)=f_0+f_1q+\ldots+f_n q^r\mbox{ and }g(x)=g_0+g_1q+\ldots+g_n q^r,
\]
with $f_i,g_i\in K[x]_n$ for every $i$, $0\leq i\leq r$. Then
\begin{displaymath}
\begin{array}{rcl}
\mu'(f+g)&= &\displaystyle\min_{0\leq i\leq r}\{\min\{\mu(f_i+g_i)\}+i\gamma\}\\[8pt]
&\geq&\displaystyle\min_{0\leq i\leq r}\{\min\{\mu(f_i),\mu(g_i)\}+i\gamma\}\\[8pt]
&=&\displaystyle\min\left\{\min_{0\leq i\leq r}\{\mu(f_i)+i\gamma\},\min_{0\leq i\leq r}\{\mu(g_i)+i\gamma\}\right\}\\[8pt]
&=&\min\{\mu'(f),\mu'(g)\}.
\end{array}
\end{displaymath}
Hence, \textbf{(V2)} is satisfied for $\mu'$.
\end{proof}

\begin{proof}[Proof of Theorem \ref{lemasobreooutrocoiso}]
Take $f,g\in K[x]$ and consider their $q$-expansions
\[
f=f_0+\ldots+f_rq^r\mbox{ and }g=g_0+\ldots+g_sq^s.
\]
In particular, for every $i$, $0\leq i\leq r$, and $j$, $0\leq j\leq s$ we have
\begin{equation}\label{equjajudanolemadoi}
\mu(f_i)\geq \mu'(f)-i\gamma\mbox{ and }\mu(g_j)\geq \mu'(g)-j\gamma.
\end{equation}
For each $i$, $0\leq i\leq r$, and $j$, $0\leq j\leq s$, let
\[
f_ig_j=a_{ij}q+c_{ij}
\]
be the $q$-standard expansion of $f_ig_j$.

If $r=0=s$, then by our assumptions
\begin{displaymath}
\begin{array}{rcl}
\mu'(fg)&=&\min\{\mu(c_{00}),\mu(a_{00})+\gamma\}=\mu(c_{00})=\mu(f_0g_0)\\[8pt]
&=&\mu(f_0)+\mu(g_0)=\mu'(f)+\mu'(g).
\end{array}
\end{displaymath}

Since $\mu$ satisfies \textbf{(V2)}, by Lemma \ref{Lemasobreextsvelidesitra}, also $\mu'$ satisfies \textbf{(V2)}. Hence
\[
\mu'(fg)\geq\min_{i,j}\{\mu'(f_ig_jq^{i+j})\}=\min_{i,j}\{\mu'(f_iq^i)+\mu'(g_jq^j)\}=\mu'(f)+\mu'(g).
\]
For each $i$, $0\leq i\leq r$, and $j$, $0\leq j\leq s$, our assumptions give
\[
\mu'(f_iq^i)+\mu'(g_jq^j)=\mu(f_ig_j)+(i+j)\gamma=\nu(c_{ij})+(i+j)\gamma=\mu'(c_{ij}q^{i+j}).
\]
Let $i_0$, $0\leq i_0\leq r$ and $j_0$, $0\leq j_0\leq s$, be the smallest non-negative integers such that the equality holds in \eqref{equjajudanolemadoi}. Set $k_0:=i_0+j_0$. Then for every $i<k_0$ if $i\neq i_0$, then
\[
\mu'(f)<\mu(f_i)-i\gamma\mbox{ or } \mu'(g)<\mu(g_{k_0-i})-(k_0-i)\gamma.
\]
Then
\begin{displaymath}
\begin{array}{rcl}
\mu(c_{i_0j_0})&=&\mu(f_{i_0})+\mu(g_{j_0})\\[8pt]
&= &\mu'(f)+i_0\gamma+\mu'(g)+j_0\gamma=\mu'(f)+\mu'(g)+k_0\gamma\\[8pt]
&<&\mu(f_i)+\nu(g_{k_0-i})-k_0\gamma+k_0\gamma=\mu(f_ig_{k_0-i})=\mu(c_{i(k_0-i)}).
\end{array}
\end{displaymath}
On the other hand, for any $i\leq k_0-1$ we have
\begin{displaymath}
\begin{array}{rcl}
\mu(a_{i(k_0-i-1)})&>&\mu(c_{i(k_0-i-1)})-\gamma=\mu(f_ig_{k_0-i-1})-\gamma\\[8pt]
&=& \mu(f_i)+\mu(g_{k_0-i-1})-\gamma\\[8pt]
&\geq& \mu'(f)+\mu'(g)-(k_0-1)\gamma-\gamma\\[8pt]
&=&\mu'(f)+\mu'(g)-k_0\gamma=\mu(c_{i_0j_0})+k_0\gamma-k_0\gamma\\[8pt]
&=&\mu(c_{i_0j_0}).
\end{array}
\end{displaymath}

Let $fg=a_0+a_1q+\ldots+a_lq^l$ be the $q$-standard expansion of $fg$. Then
\[
a_{k_0}=\sum_{i=0}^{k_0}c_{i(k_0-i)}+\sum_{i=0}^{k_0-1}a_{i(k_0-i-1)}.
\]
Since $\mu(c_{i_0j_0})<\mu(a_{i(k_0-i-1)})$ for $i\leq k_0-1$ and $\mu(c_{i_0j_0})<\mu(c_{i(k_0-i)})$ for $i\neq i_0$ we have
\[
\mu(a_{k_0})=\mu(c_{i_0j_0})=\mu(f_{i_0})+\mu(g_{j_0})=\mu'(f)+\mu'(g)-k_0\gamma.
\]
Here we are using that since $\mu$ satisfies \textbf{(V2)} (by the previous lemma), if $\mu(F)<\mu(G)$ for $F,G\in K[x]$, then $\mu(F+G)=\mu(F)$.

Therefore,
\[
\mu'(fg)=\min_{0\leq k\leq l}\{\mu(a_k)+k\gamma)\}\leq \mu'(f)+\mu'(g),
\]
which completes the proof.
\end{proof}
Given a valuation $\nu_0$ on $K$, a group $\Gamma'$ containing $\nu_0 K$ and $\gamma\in \Gamma'$ we define
\[
\nu_\gamma(a_0+a_1x+\ldots+a_nx^n):=\min_{0\leq i\leq n}\{\nu_0(a_i)+i\gamma\}.
\]
\begin{Cor}\label{Thmmonval}
The map $\nu_\gamma$ is a valuation on $K[x]$.
\end{Cor}
\begin{proof}
The property \textbf{(V3)} follows from the definition of $\nu_\gamma$ and property \textbf{(V2)} is a consequence of Lemma \ref{Lemasobreextsvelidesitra}. In order to prove \textbf{(V1)} we will use Theorem \ref{lemasobreooutrocoiso} for $\mu=\mu'=\nu_\gamma$. Since $\deg(x)=1$ and $K[x]_1=K$ the condition \textbf{(i)} of Theorem \ref{lemasobreooutrocoiso} is satisfied (because $K[x]_1=K$ is closed by multiplication and $\nu\mid_K=\nu_0$ is a valuation). Also, for $f,g\in K$ if $fg=ax+r$, then $a=0$ and $r=fg$ because $x$ is transcendental over $K$. Hence
\[
\nu(fg)=\nu(r)<\infty=\nu(aQ)
\]
and therefore the condition \textbf{(ii)} of Theorem \ref{lemasobreooutrocoiso} is satisfied. Hence, \textbf{(V1)} is satisfied for $\mu'=\nu_\gamma$ and consequently $\nu_\gamma$ is a valuation.
\end{proof}
\begin{Def}
The valuations of $K[x]$ that are constructed as the previous theorem are called \textbf{monomial valuations} and are denoted by
\[
\nu_\gamma:=[\nu_0;\nu_\gamma(x)=\gamma].
\]
\end{Def}

Many results follow as corollaries of Theorem \ref{lemasobreooutrocoiso}. Examples of this are Proposition \ref{proptruncakeypolval}, Theorem \ref{Thmaugmentedval} and Theorem \ref{Theoremsobreexntvalaugvaeui} in this paper.

\begin{Obs}
The main argument used to prove Theorem 3 of \cite{Kap} follows from Theorem \ref{lemasobreooutrocoiso}. This result refers to \textit{pseudo-convergent sequences} and \textit{immediate extension}. Since this topic is not necessary in this paper, we will not present it here.
\end{Obs}
\section{Key polynomials}\label{keypolynomials}
In order to define a key polynomial, we will need to define the number $\epsilon(f)$ for $f\in K[x]$. Let $\Gamma'=\Gamma\otimes \Q$ be the divisible hull of $\Gamma$. For a polynomial $f\in K[x]$ and $k\in\N$, we consider
\[
\partial_k(f):=\frac{1}{k!}\frac{d^kf}{dx^k},
\]
the so called Hasse-derivative of $f$ of order $k$. Let
\[
\epsilon(f)=\max_{k\in \N}\left\{\frac{\nu(f)-\nu(\partial_kf)}{k}\right\}\in \Gamma'.
\]
\begin{Def}
A monic polynomial $Q\in K[x]$ is said to be a \textbf{key polynomial} (of level $\epsilon (Q)$) if for every $f\in K[x]$ if $\epsilon(f)\geq \epsilon(Q)$, then $\deg(f)\geq\deg(Q)$.
\end{Def}

\begin{Lema}[Lemma 2.3 of \cite{SopivNova}]
Let $Q$ be a key polynomial and take $f,g\in K[x]$ such that
\[
\deg(f)<\deg(Q)\mbox{ and }\deg(g)<\deg(Q).
\]
Then for $\epsilon:=\epsilon(Q)$ and any $k\in\N$ we have the following:
\begin{description}\label{lemaonkeypollder}
\item[(i)] $\nu(\partial_k(fg))>\nu(fg)-k\epsilon$
\item[(ii)] If $\nu_Q(fQ+g)<\nu(fQ+g)$ and $k\in I(Q):=\left\{i\mid \epsilon(f)=\frac{\nu(f)-\nu(\partial_if)}{i}\right\}$, then $\nu(\partial_k(fQ+g))=\nu(fQ)-k\epsilon$;
\item[(iii)] If $h_1,\ldots,h_s$ are polynomials such that $\deg(h_i)<\deg(Q)$ for every $i=1,\ldots, s$ and
$\displaystyle\prod_{i=1}^sh_i=qQ+r$ with $\deg(r)<\deg(Q)$ and $r\neq 0$, then
\[
\nu(r)=\nu\left(\prod_{i=1}^sh_i\right)<\nu(qQ).
\]
\end{description}
\end{Lema}

\begin{Prop}[Proposition 2.4 \textbf{(ii)} of \cite{SopivNova}]
Every key polynomial is irreducible.
\end{Prop}
The next result is Proposition 2.6 of \cite{SopivNova} (which is the same as Proposition 15 of \cite{Spivmahandjul}). We present its proof here because it follows easily from Theorem \ref{lemasobreooutrocoiso}.
\begin{Prop}\label{proptruncakeypolval}
If $Q$ is a key polynomial, then $\nu_Q$ is a valuation of $K[x]$.
\end{Prop}
\begin{proof}
The property \textbf{(V3)} follows from the definition of $\nu_Q$ and property \textbf{(V2)} is a consequence of Lemma \ref{Lemasobreextsvelidesitra}. In order to prove property \textbf{(V1)}, we will use Lemma \ref{lemasobreooutrocoiso} for $\mu=\nu$ and $\mu'=\nu_Q$. The conditions \textbf{(i)} of Lemma \ref{lemasobreooutrocoiso} is satisfied because $\nu$ is a valuation. We observe that since $Q$ is irreducible if $fg=aQ+r$, then $r\neq 0$. Hence, the condition \textbf{(ii)} of Lemma \ref{lemasobreooutrocoiso} follows from Lemma \ref{lemaonkeypollder} \textbf{(iii)}.
\end{proof}

\begin{Prop}[Proposition 2.10 of \cite{SopivNova}]\label{Propcompkeypol}
For two key polynomials $Q,Q'\in K[x]$ we have the following:
\begin{description}
\item[(i)] If $\deg(Q)<\deg(Q')$, then $\epsilon(Q)<\epsilon(Q')$;
\item[(ii)] If $\epsilon(Q)<\epsilon(Q')$, then $\nu_Q(Q')<\nu(Q')$;
\item[(iii)] If $\deg(Q)=\deg(Q')$, then
\begin{equation}\label{eqwhdegsame}
\nu(Q)<\nu(Q')\Llr \nu_Q(Q')<\nu(Q')\Llr \epsilon(Q)<\epsilon(Q').
\end{equation}
\end{description}
\end{Prop}
\begin{Cor}\label{corsobreequaldepsoidkey}
Let $Q$ and $Q'$ be key polynomials such that $\epsilon(Q)\leq\epsilon(Q')$. For every $f\in K[x]$, if $\nu_Q(f)=\nu(f)$, then $\nu_{Q'}(f)=\nu(f)$.
\end{Cor}
\begin{proof}
It follows from Proposition \ref{Propcompkeypol} that if $\epsilon(Q)\leq\epsilon(Q')$, then $\nu_{Q'}(Q)=\nu(Q)$. Since $\deg(Q)\leq \deg(Q')$, for every $f_i\in K[x]$ with $\deg(f_i)<\deg(Q)$ we have $\nu_{Q'}(f_i)=\nu(f_i)$. Hence $\nu_{Q'}(f_iQ^i)=\nu(f_iQ^i)$.

Take $f\in K[x]$ such that $\nu_Q(f)=\nu(f)$ and let
\[
f=f_0+f_1 Q+\ldots+f_nQ^n
\]
be the $Q$-expansion of $f$. Then
\[
\nu_{Q'}(f)\geq\min_{0\leq i\leq n}\{\nu_{Q'}(f_iQ^i)\}=\min_{0\leq i\leq n}\{\nu(f_iQ^i)\}=\nu_Q(f)=\nu(f).
\]
Since $\nu_{Q'}(f)\leq \nu(f)$ for every $f\in K[x]$ we have our result.
\end{proof}
For a key polynomial $Q\in K[x]$, let
\[
\alpha(Q):=\min\{\deg(f)\mid \nu_Q(f)< \nu(f)\}
\]
(if $\nu_Q=\nu$, then set $\alpha(Q)=\infty$) and
\[
\Psi(Q):=\{f\in K[x]\mid f\mbox{ is monic},\nu_Q(f)< \nu(f)\mbox{ and }\alpha(Q)=\deg (f)\}.
\]

\begin{Teo}[Theorem 2.12 of \cite{SopivNova}]\label{definofkeypol}
A monic polynomial $Q$ is a key polynomial if and only if there exists a key polynomial $Q_-\in K[x]$ such that $Q\in \Psi(Q_-)$ or the following conditions hold:
\begin{description}
\item[(K1)] $\alpha(Q_-)=\deg (Q_-)$
\item[(K2)] the set $\{\nu(Q')\mid Q'\in\Psi(Q_-)\}$ does not contain a maximal element
\item[(K3)] $\nu_{Q'}(Q)<\nu(Q)$ for every $Q'\in \Psi(Q_-)$
\item[(K4)] $Q$ has the smallest degree among polynomials satisfying \textbf{(K3)}.
\end{description}
\end{Teo}

\begin{Def}
When conditions \textbf{(K1) - (K4)} of Theorem \ref{definofkeypol} are satisfied, we say that $Q$ is a \textbf{limit key polynomial}.
\end{Def}

\begin{Def}
A set $\textbf{Q}\subseteq K[x]$ is called a \textbf{complete set for $\nu$} if for every $f\in K[x]$ there exists $Q\in \textbf{Q}$ with $\deg(Q)\leq\deg(f)$ such that $\nu_Q(f)=\nu(f)$. If the set \textbf{Q} admits an order under which it is well-ordered, then it is called a complete sequence.
\end{Def}

\begin{Teo}[Theorem 1.1 of \cite{SopivNova}]\label{Theoremexistencecompleteseqkpol}
Every valuation $\nu$ on $K[x]$ admits a complete set $\textbf{Q}$ of key polynomials. Moreover, $\textbf{Q}$ can be chosen to be well-ordered with respect to the order given by $Q<Q'$ if $\epsilon(Q)<\epsilon(Q')$.
\end{Teo}
\begin{Obs}
In \cite{SopivNova}, the definition of complete sequence does not require that $\deg(Q)\leq\deg(f)$ as in our definition. This property is important and the proof of Theorem 1.1 in \cite{SopivNova} guarantees that the obtained sequence satisfies the additional property.
\end{Obs}
\section{Graded algebras}\label{gradedalgebras}
Let $R$ be an integral domain and $\nu$ a valuation on $R$. For each $\gamma\in \nu(R)$, we consider the abelian groups
\begin{equation*}
\mathcal{P}_{\gamma}=\{a \in R \mid \nu(a) \geq \gamma\} \mbox{ and } \mathcal{P}^{+}_{\gamma}=\{a \in R \mid \nu(a) > \gamma\}.
\end{equation*}

\begin{Def}
The \textbf{graded algebra of $R$ associated to $\nu$} is defined as
\begin{equation*}
{\rm gr}_\nu(R)= \displaystyle  \bigoplus_{\gamma \in \nu(R)} \mathcal{P}_{\gamma}/\mathcal{P}^{+}_{\gamma}.
\end{equation*}
\end{Def}
It is not difficult to show that ${\rm gr}_\nu(R)$ is an integral domain. For $a\in R$ we will denote by ${\rm in}_\nu(a)$ the image of $a$ in
\[
\mathcal P_{\nu(a)}/\mathcal P_{\nu(a)}^+\subseteq{\rm gr}_\nu(R).
\]

For the remaining of this section we will consider a valuation $\nu$ on $K[x]$ and a key polynomial $Q\in K[x]$. Then the truncation $\nu_Q$ of $\nu$ on $Q$ is a valuation. For simplicity of notation we will write
\[
{\rm gr}_{Q}(K[x]):={\rm gr}_{\nu_Q}(K[x])
\]
and for $f\in K[x]$ we denote $\inv_Q(f):= \inv_{\nu_Q}(f)$. Let $R_Q$ be the additive subgroup of ${\rm gr}_{Q}(K[x])$ generated by
\[
\{\inv_Q(f)\mid f\in K[x]_n\}\mbox{ where }n:=\deg(Q),
\]
and
\[
y:=\inv_Q(Q).
\]
\begin{Prop}
The set $R_Q$ is a subring of ${\rm gr}_{Q}(K[x])$, $y$ is transcendental over $R_Q$ and
\[
{\rm gr}_{Q}(K[x])=R_Q[y].
\]
\end{Prop}
\begin{proof}
In order to prove that $R_Q$ is a subgring it is enough to show that it is closed under addition and multiplication, i.e., for $\phi,\psi\in R_Q$ we have that $\phi+\psi\in R_Q$ and $\phi\cdot\psi\in R_Q$. Since $R_Q$ is defined as an additive subgrop, it is closed by addition. In order to show that it is closed under multiplication, it is enough to consider $\phi=\inv_Q(f)$ and $\psi=\inv_Q(g)$ for some $f,g\in K[x]_n$. Write
\[
fg=aQ+r\mbox{ for }a,r\in K[x]_n.
\]
Since $Q$ is irreducible and $\deg(f),\deg(q)<\deg(Q)$ we have that $r\neq 0$. Hence, by Proposition \ref{lemaonkeypollder} \textbf{(iii)} (for $s=2$) we have that
\[
\nu(fg)=\nu(r)<\nu(aQ).
\]
This implies that
\[
\nu_Q(fg-r)=\nu_Q(aQ)=\nu(aQ)>\nu(r)=\nu_Q(r),
\]
and consequently
\[
\inv_Q(f)\cdot\inv_Q(g)=\inv_Q(fg)=\inv_Q(r)\in R_Q.
\]
Hence, $R_Q$ is a ring.

To prove that $y$ is transcendental over $R_Q$ assume that there exists an algebraic equation
\begin{equation}\label{eqsobrelevelrpollingradalg}
a_0+a_1y+\ldots+a_ry^r=0
\end{equation}
for $a_i\in R_Q$. We can assume that for each $i$, $0\leq i\leq n$, $a_i$ is homogeneous of the form $a_i=\inv_Q(f_i)$ and that all the terms on the left hand side in \eqref{eqsobrelevelrpollingradalg} have the same degree. Then
\[
\nu(f_0)=\nu_Q(f_0)=\nu_Q(f_iQ^i)=\nu(f_iQ^i)\mbox{ for every }i,0\leq i\leq r.
\]
This and \eqref{eqsobrelevelrpollingradalg} imply that
\[
\min_{0\leq i\leq n}\{\nu(f_iQ^i)\}=\nu(f_0)=\nu_Q(f_0)<\nu_Q(f_0+f_1Q+\ldots+f_rQ^r)
\]
and this is a contradiction to the definition of $\nu_Q$.

Take now any polynomial $f\in K[x]$ and write $f=f_0+f_1Q+\ldots+f_rQ^r$ with $f_i\in K[x]_n$ for every $i$, $0\leq i\leq r$. Then
\[
\nu_Q\left(f-\sum_{i\in S_Q(f)}f_iQ^i\right)=\nu_Q\left(\sum_{i\notin S_Q(f)}f_iQ^i\right)=\min_{i\notin S_Q(f)}\{\nu\left(f_iQ^i\right)\}>\nu_Q(f)
\]
(where $S_Q(f):=\{i\in\{0,\ldots,r\}\mid \nu(f_iQ^i)=\nu_Q(f)\}$).
Hence,
\[
\inv_Q(f)=\sum_{i\in S_Q(f)}\inv_Q(f_i)y^i\in R_Q[y].
\]
This concludes the proof.
\end{proof}
\begin{Obs}
We observe that the property of key polynomials used to prove the previous result is that they satisfy Property \textbf{(ii)} of Theorem \ref{lemasobreooutrocoiso}.
\end{Obs}
\begin{Cor}\label{Corlsobremaclaneparatruncame}
For $f,g\in K[x]$ if $y$ divides $\inv_Q(f)\cdot\inv_Q(g)$, then $y$ divides $\inv_Q(f)$ or $\inv_Q(g)$ in ${\rm gr}_Q(K[x])$.
\end{Cor}
\begin{proof}
Since $R_Q$ is a domain and $y$ is transcendental over $R_Q$, we have that $y$ is irreducible in ${\rm gr}_Q(K[x])=R_Q[y]$. Our result follows immediately.
\end{proof}
\begin{Cor}\label{Corlsobremaclaneparatruncame1}
For $f\in K[x]$, if $y$ divides $\inv_Q(f)$ in ${\rm gr}_Q(K[x])$, then $\deg(Q)\leq \deg(f)$.
\end{Cor}
\begin{proof}
If $\deg(f)<\deg(Q)$, then $\inv_Q(f)\in R_Q$. Hence, $y\nmid \inv_Q(f)$.
\end{proof}

For the remaining of this section, we will consider a key polynomial $Q$ for $\nu$ and fix an element $Q'\in \Psi(Q)$. We want to study the properties of $\inv_Q(Q')$ in ${\rm gr}_Q(K[x])$. We will start with the following basic result.

\begin{Lema}\label{Lemanotexpectedabihomele}
Let $R$ be a ring, graded by a totally ordered semigroup $\Gamma$. Let $I=\langle f\rangle$ be an ideal generated by a homogeneous element $f$. If for any homogeneous elements $g,h\in R$ we have
\[
f\mid gh\Lra f\mid g\mbox{ or }f\mid h,
\]
then $I$ is a prime ideal.
\end{Lema}
\begin{proof}
Let $g=g_1+\ldots+g_r$ and $h=h_1+\ldots+h_s$ where $g_i$ and $h_j$ are homogeneous, $\deg(g_i)<\deg(g_{i+1})$ and $\deg(h_j)<\deg(h_{j+1})$ for every $i$, $1\leq i\leq r-1$ and $j$, $1\leq j\leq s-1$. Assume that $gh\in I$, i.e., that $f\mid gh$. We will prove first that
\begin{equation}\label{oureeqkhaprecemastrivial}
f\mid g_i\mbox{ or }f\mid h_j\mbox{ for each }i, 1\leq i\leq r\mbox{ and }j, 1\leq j\leq s.
\end{equation}
Since $\deg(g_1)<\deg(g_i)$ and $\deg(h_1) <\deg(h_j)$ for $i>1$ and $j>1$, the fact that $f\mid gh$ implies that $f\mid g_1h_1$. By our assumption we have that $f\mid g_1$ or $f\mid h_1$. For a given $i_0$, $1<i_0\leq r$ and $j_0$, $1<j_0\leq s$, assume that
\begin{equation}\label{equationofgradedlouco}
f\mid g_{i}\mbox{ or }f\mid h_{j}\mbox{ for each }i<i_0\mbox{ and }j<j_0.
\end{equation}
Since $f\mid gh$, we have that $f$ divides $\displaystyle\sum_{i+j=i_0+j_0}g_ih_j$. By \eqref{equationofgradedlouco} we conclude that $f\mid g_{i_0}h_{j_0}$ and by our assumption $f\mid g_{i_0}$ or $f\mid h_{j_0}$. Recursively, we obtain \eqref{oureeqkhaprecemastrivial}.

It remains to show that $f\mid g$ or $f\mid h$, i.e., that $f\mid g_i$ for every $i$, $1\leq i\leq r$ or that $f\mid h_j$ for every $j$, $1\leq j\leq s$. If this were not the case, then there would exist $i$, $1\leq i\leq r$ and $j$, $1\leq j\leq s$ such that $f\nmid g_i$ and $f\nmid h_j$. This is a contradiction to \eqref{oureeqkhaprecemastrivial}.
\end{proof}

\begin{Lema}\label{lemaqueajudanapropsdoirredl}
For $f\in K[x]$, $\inv_Q(Q')\mid\inv_Q(f)$ if and only if $\nu_Q(f)<\nu(f)$. In particular, if $\inv_Q(Q')\mid\inv_Q(f)$, then $\deg(f)\geq \deg(Q')$.
\end{Lema}
\begin{proof}
Take any polynomial $f\in K[x]$ such that $\nu_Q(f)<\nu(f)$. By the minimality of the degree of $Q'$, we have that
\[
f=a Q'+r\mbox{ with }a\neq 0\mbox{ and }\deg(r)<\deg(Q').
\]
Since $\deg(r)<\deg(Q')$ we have that $\nu_Q(r)=\nu(r)$.
Then
\[
\nu_Q(r)=\nu(r)\geq\min\{\nu(f),\nu(aQ')\}>\min\{\nu_Q(f),\nu_Q(aQ')\},
\]
which gives us
\[
\inv_Q(f)=\inv_Q(aQ')=\inv_Q(a)\inv_Q(Q').
\]

For the converse, assume that $\inv_Q(Q')\mid\inv_Q(f)$. This implies that there exists $a\in K[x]$ such that $\inv_Q(aQ')=\inv_Q(f)$. Then
\[
\nu_Q(f-aQ')>\nu_Q(f)=\nu_Q(aQ').
\]
If $\nu_Q(f)=\nu(f)$, then
\[
\nu(f-aQ')\geq\nu_Q(f-aQ')>\nu_Q(f)=\nu(f).
\]
This implies that $\nu(f)=\nu(aQ')$ and hence
\[
\nu(aQ')=\nu(f)=\nu_Q(f)=\nu_Q(aQ').
\]
This is a contradiction to $\nu(Q')<\nu_Q(Q')$.
\end{proof}
\begin{Prop}
Assume that $Q$ is a key polynomial and that $Q'\in\Psi(Q)$. Then
\[
I_Q:=\left\langle\{\inv_Q(f)\mid\nu_Q(f)<\nu(f)\}\right\rangle
\]
is a prime ideal of ${\rm gr}_Q(K[x])$ generated by $\inv_Q(Q')$.
\end{Prop}
\begin{proof}
For each $f\in K[x]$ with $\nu_Q(f)<\nu(f)$, Lemma \ref{lemaqueajudanapropsdoirredl} gives us that $\inv_Q(Q')\mid\inv_Q(f)$. This implies that $I_Q\subseteq \langle\inv_Q(Q')\rangle$ and since the other inclusion is trivial, we have the equality.

By Lemma \ref{Lemanotexpectedabihomele}, in order to prove that $I_Q$ is a prime ideal, it is enough to show that for $f,g\in K[x]$, if $\inv_Q(Q')\mid\inv_Q(f)\cdot \inv_Q(g)$, then
\[
\inv_Q(Q')\mid\inv_Q(f)\mbox{ or }\inv_Q(Q')\mid \inv_Q(g).
\]
Assume that
\[
\inv_Q(Q')\mid\inv_Q(f)\cdot \inv_Q(g)=\inv_Q(fg).
\]
By Lemma \ref{lemaqueajudanapropsdoirredl} we have that $\nu_Q(fg)<\nu(fg)$ and hence $\nu_Q(f)<\nu(f)$ or $\nu_Q(g)<\nu(g)$. Hence, by Lemma \ref{lemaqueajudanapropsdoirredl}
\[
\inv_Q(Q')\mid\inv_Q(f)\mbox{ or }\inv_Q(Q')\mid\inv_Q(g).
\]
This concludes the proof.
\end{proof}
\begin{Cor}\label{Corlsobremaclaneparatruncame3}
The element $\inv_Q(Q')$ is irreducible in ${\rm gr}_Q(K[x])$.
\end{Cor}

\section{MacLane-Vaqui\'e key polynomials and augmented valuations}\label{Maclanevaquie}
This section is based on \cite{Mac_1} and \cite{Vaq_1}. We decided to show all the needed results here for the convenience of the reader.
Our main goal is to axiomatize a way to extend, if possible, any given valuation on $K[x]$. For this purpose we will need to introduce the concept of MacLane-Vaqui\'e key polynomials.

Let $K$ be a field and let $\nu$ be a valuation on $K[x]$
\begin{Def}
Take $f,g\in K[x]$,
\begin{description}
\item[(i)]  We say that \textbf{$f$ is $\nu$-equivalent to $g$} (and denote by $f\sim_\nu g$) if $\textrm{in}_\nu(f)=\textrm{in}_\nu(g)$. 
\item[(ii)] We say that \textbf{$g$ $\nu$-divides $f$} (denote by $g\mid_\nu f$) if there exists $h\in K[x]$ such that $f\sim_\nu g\cdot h$.
\end{description}
\end{Def}
\begin{Def}
A monic polynomial $Q\in K[x]$ is a \textbf{MacLane-Vaqui\'e key polynomial for $\nu$} if
\begin{description}
\item[(KP1)] $Q$ is $\nu$-irreducible, i.e.,
\[
Q\mid_\nu f\cdot g \Lra Q\mid_\nu f\mbox{ or }Q\mid_\nu g;
\]
and
\item[(KP2)] for every $f\in K[x]$ we have
\[
Q\mid_\nu f\Lra \deg(f)\geq \deg(Q).
\]
\end{description}
\end{Def}
Let $Q$ be a MacLane-Vaqui\'e key polynomial for $\nu$, $\Gamma'$ be a group extension of $\nu(K[x])$ and $\gamma\in \Gamma'$ such that $\gamma>\nu(Q)$. For every $f\in K[x]$, let
\[
f=f_0+f_1Q+\ldots+f_nQ^n
\]
be the $Q$-expansion of $f$. Define the map
\begin{equation}\label{eqquedefinagungemntvalsu}
\nu'(f):=\min_{0\leq i\leq n}\{\nu(f_i)+i\gamma\}.
\end{equation}

\begin{Teo}\label{Thmaugmentedval}
The map $\nu'$ is a valuation on $K[x]$.
\end{Teo}
In order to prove Theorem \ref{Thmaugmentedval}, we will need the following lemma.

\begin{Lema}\label{primeiraparteleloup}
Let $Q$ be a MacLane-Vaqui\'e key polynomial for $\nu$. Then
\begin{description}
\item[(i)] for $f,r,a\in K[x]$ with $\deg(r)<\deg(Q)$ and $f=aQ+r$ we have
\[
\nu(f)\leq \min\{\nu(aQ),\nu(r)\}; \mbox{ and}
\]
\item[(ii)] for $f,g,r,a\in K[x]$ with
\[
\max\{\deg(f),\deg(g),\deg(r)\}<\deg(Q)
\]
if $fg=aQ+r$, then
\begin{equation}\label{aeququenaoepsjracont}
\nu(r)=\nu(fg)<\nu(a)+\gamma.
\end{equation}
\end{description}
\end{Lema}
\begin{proof}
In order to prove \textbf{(i)} we observe that if $\nu(f)>\nu(r)$, then
\[
\nu(r)<\nu(f)=\nu(aQ+r).
\]
Consequently $Q\mid_\nu -r$ and since $\deg(r)<\deg(Q)$ this contradicts the fact that $Q$ is a MacLane-Vaqui\'e key polynomial for $\nu$. Hence, $\nu(r)\geq\nu(f)$ and consequently
\[
\nu(aQ)\geq\min\{\nu(f),\nu(r)\}=\nu(f)
\]
and this shows \textbf{(i)}.

In order to prove \textbf{(ii)}, assume aiming for a contradiction, that \eqref{aeququenaoepsjracont} is not satisfied. Then
\begin{equation}\label{ainbdamaisdoidaeenex}
\max\{\nu(r),\nu(fg)\}>\nu(aQ).
\end{equation}
Indeed, if
\[
\nu(r)\geq \nu(a)+\gamma\mbox{ or }\nu(fg)\geq \nu(a)+\gamma,
\]
then we have \eqref{ainbdamaisdoidaeenex} because $\gamma>\nu(Q)$. On the other hand, if $\nu(r)\neq\nu(fg)$, then
\[
\max\{\nu(r),\nu(fg)\}>\min\{\nu(r),\nu(fg)\}=\nu(aQ).
\]

If $\nu(fg)> \nu(aQ)$, then
\[
\nu(aQ)=\nu(r)<\nu(aQ+r)=\nu(fg)
\]
and hence $Q\mid_\nu r$. Analogously, if $\nu(r)>\nu(aQ)$, then $Q\mid_\nu fg$ and since $Q$ satisfies \textbf{(KP1)} we conclude that $Q\mid_\nu f$ or $Q\mid_\nu g$. In each case we obtain a contradiction, because $\max\{\deg(r),\deg(f),\deg(g)\}<\deg(Q)$ and $Q$ satisfies \textbf{(KP2)}.
\end{proof}
\begin{proof}[Proof of Theorem \ref{Thmaugmentedval}]
The property \textbf{(V3)} follows from the definition of $\nu'$ and property \textbf{(V2)} is a consequence of Lemma \ref{Lemasobreextsvelidesitra}. In order to prove \textbf{(V1)} it is enough to show that conditions \textbf{(i)} and \textbf{(ii)} of Lemma \ref{lemasobreooutrocoiso} are satisfied for $\nu=\mu$ and $\nu'=\mu'$. The condition \textbf{(i)} is satisfied because $\nu$ is a valuation and condition \textbf{(ii)} follows immediately from Lemma \ref{primeiraparteleloup}.
\end{proof}

\begin{Def}
The map $\nu'$ is called an \textbf{augmented valuation} and denoted by
\[
\nu':=[\nu;\nu'(Q)=\gamma].
\]
\end{Def}
\begin{Obs}
From now on, when we say that $\nu'$ is of the form
\[
\nu'=[\nu;\nu'(Q)=\gamma]
\]
we mean that $\nu$ is a valuation on $K[x]$, $Q$ is a key polynomial for $\nu$, $\gamma>\nu(Q)$ and that $\nu'$ is the valuation presented in \eqref{eqquedefinagungemntvalsu}.
\end{Obs}

We want to iterate the construction of augmented valuations as above. For this purpose we need the concept of iterated family of valuations. Consider a family $\mathcal F=\{(\nu_\alpha,Q_\alpha,\gamma_\alpha)\}_{\alpha\in A}$, indexed by a totally ordered set $A$, where for every $\alpha\in A$, $\nu_\alpha$ is a valuation on $K[x]$, $Q_\alpha\in K[x]$ and $\gamma_\alpha$ is an element in a fixed ordered abelian group $\Gamma$.
\begin{Def}\label{defoffamiliaugmentedintereda}
The family $\mathcal F$ is called a \textbf{family of augmented iterated valuations} if for every $\alpha\in A$, except the smallest element of $A$, there exists $\alpha_-\in A$, $\alpha_-<\alpha$, such that $\nu_\alpha= [\nu_{\alpha_-};\nu_\alpha(Q_\alpha)=\gamma_\alpha]$, and the following properties hold.
\begin{description}
\item[(i)] If $\alpha$ admits an immediate predecessor in $A$, $\alpha_-$ is that predecessor, and in the case when $\alpha_-$ is not the smallest element of $A$, the polynomials $Q_\alpha$ and $Q_{\alpha_-}$ are not $\nu_{\alpha_-}$-equivalent and satisfy $\deg(Q_{\alpha_-}) \leq \deg(Q_\alpha)$;
\item[(ii)] If $\alpha$ does not have an immediate predecessor in $A$, for every $\beta\in A$ such that $\alpha_- < \beta < \alpha$, we have
\[
\nu_\beta = [\nu_{\alpha_-} ; \nu_\beta (Q_\beta) = \gamma_\beta]
\]
and
\[
\nu_\alpha = [\nu_\beta; \nu_\alpha(Q_\alpha) =\gamma_\alpha ],
\]
and the polynomials $Q_\alpha$ and $Q_{\alpha_-}$ have the same degree.
\end{description}
\end{Def}

\begin{Obs}
If $\mathcal F=\{(\nu_\alpha,Q_\alpha,\gamma_\alpha)\}_{\alpha\in A}$ is a family of augmented iterated valuations of $K[x]$ and $I$ is a final or initial segment of $A$, then also
\[
\mathcal F_I=\{(\nu_\alpha,Q_\alpha,\gamma_\alpha)\}_{\alpha\in I}
\]
is a family of augmented iterated valuations of $K[x]$.
\end{Obs}

\begin{Def}
For a family of augmented valuations $\mathcal F=\{(\nu_\alpha,Q_\alpha,\gamma_\alpha)\}_{\alpha\in A}$ and polynomials $f,g\in K[x]$ we say that $f$ $\mathcal F$-divides $g$, and denote by $f\mid_{\mathcal F} g$, if there exists $\alpha_0\in A$ such that $f\mid_{\nu_\alpha} g$ for every $\alpha\in A$ with $\alpha\geq\alpha_0$.
\end{Def}

\begin{Def}
A family of augmented iterated valuations
\[
\mathcal F=\{(\nu_\alpha,Q_\alpha,\gamma_\alpha)\}_{\alpha\in A}
\]
is said to be continued if $\{\gamma_\alpha\mid\alpha\in A\}$ does not have maximal element, $\deg(Q_\alpha)=\deg(Q_\beta)$ for every $\alpha,\beta\in A$ and there exists a valuation $\nu$ on $K[x]$ such that
\[
\nu_\alpha= [\nu; \nu_\alpha(Q_\alpha)=\gamma_\alpha]\mbox{ for every }\alpha\in A.
\]
\end{Def}

\begin{Lema}\label{lemaesquscpamimport}
Let $\mathcal{F}$ be a continued family of augmented valuations and assume that for every $\alpha,\beta\in A$ with $\alpha<\beta$, we have $\nu_\alpha\leq\nu_\beta$, i.e., $\nu_\alpha(f)\leq \nu_\beta(f)$ for every $f\in K[x]$. Then for $f\in K[x]$ we have that
\[
\nu_\alpha(f)<\nu_\beta(f)\mbox{ for every }\alpha,\beta\in A\mbox{ with }\alpha<\beta
\]
or there exists $\alpha_f\in A$ such that
\[
\nu_{\alpha_f}(f)=\nu_\alpha(f)\mbox{ for every }\alpha\geq\alpha_f.
\]
\end{Lema}
\begin{proof}
Assume that $\nu_\alpha(f)=\nu_\beta(f)$ for some $\alpha,\beta\in A$. We claim that $\alpha$ and $\beta$ can be choosen in a way that
\begin{equation}\label{equantosbreaugmevalutin}
\nu_\beta=[\nu_\alpha;\nu_\beta(Q_\beta)=\gamma_\beta].
\end{equation}
Indeed, for every $\alpha'\in A$ with $\alpha\leq \alpha'\leq \beta$, we have
\[
\nu_\alpha(f)\leq \nu_{\alpha'}(f)\leq \nu_\beta(f)=\nu_\alpha(f)
\]
and hence the equality holds everywhere. If $\alpha\geq \beta_-$, then the fact that $\mathcal F$ is a family of iterated valuations implies that \eqref{equantosbreaugmevalutin} is satisfied. If $\alpha< \beta_-$, we replace $\alpha$ by $\beta_-$ and also have \eqref{equantosbreaugmevalutin}.

We will show that for every $\beta'>\beta$ we have $\nu_{\beta'}(f)=\nu_\beta(f)$. Write $f=aQ_\beta+r$ with $\deg(r)<\deg(Q_\beta)$. Then by Lemma \ref{aeququenaoepsjracont} \textbf{(ii)} we have
\begin{equation}\label{eugostosdeisa1}
\nu_\alpha(aQ_\beta)\geq \nu_\alpha(f).
\end{equation}
Also, since $\nu_\beta=[\nu_\alpha;\nu_\beta(Q_\beta)=\gamma_\beta]$ and $\nu_\alpha\leq\nu_\beta$, we have
\begin{equation}\label{eugostosdeisa2}
\nu_\beta(aQ_\beta)>\nu_\alpha(aQ_\beta).
\end{equation}
Puting \eqref{eugostosdeisa1} and \eqref{eugostosdeisa2} together we obtain
\begin{equation}\label{maisnumaehqojsjdhka}
\nu_\beta(f-r)=\nu_\beta(aQ_\beta)>\nu_\alpha(aQ_\beta)\geq \nu_\alpha(f)=\nu_\beta(f).
\end{equation}
In particular, $\nu_\beta(r)=\nu_\beta(f)$. Since $\deg(r)<\deg(Q_\beta)=\deg(Q_{\beta'})$ we have
\begin{equation}\label{eugostosdeisa3}
\nu(r)=\nu_{\beta}(r)=\nu_{\beta'}(r).
\end{equation}
Hence, by \eqref{maisnumaehqojsjdhka}, \eqref{eugostosdeisa3} and the fact that $\nu_\beta\leq\nu_{\beta'}$, we have
\[
\nu_{\beta'}(f-r)\geq\nu_\beta(f-r)>\nu_\beta(f)=\nu_{\beta'}(r)
\]
and consequently
\[
\nu_{\beta'}(f)=\nu_{\beta'}(r)=\nu_\beta(r)=\nu_\beta(f).
\]
\end{proof}

For a continued family of iterated valuations $\mathcal F$, we define the set
\[
\overline{\Phi}(\mathcal F):=\{f\in K[x]\mid \nu_\alpha(f)<\nu_\beta(f)\mbox{ for every }\alpha,\beta\in A\mbox{ with }\alpha<\beta\}.
\]
and
\[
\mathcal C(\mathcal F):=\{f\in K[x]\mid \exists \alpha_f\in A\mbox{ such that }\nu_\alpha(f)=\nu_{\alpha_f}(f),\forall\alpha\geq\alpha_f\}.
\]
\begin{Obs}
Lemmma \ref{lemaesquscpamimport} tells us that if $\mathcal F$ is a continued family of iterated valuations, with $\nu_\alpha\leq\nu_\beta$ if $\alpha<\beta$, then
\[
K[x]= \overline{\Phi}(\mathcal F)\sqcup\mathcal C(\mathcal F)
\]
\end{Obs}

\begin{Cor}\label{corquepareceprimosnaksk}
Under the assumptions of Lemma \ref{lemaesquscpamimport}, if $fg\in \overline{\Phi}(\mathcal F)$, then $f\in \overline{\Phi}(\mathcal F)$ or $g\in \overline{\Phi}(\mathcal F)$.
\end{Cor}
\begin{proof}
Assume that $fg\in\overline\Phi(\mathcal F)$. This means that
\[
\nu_\alpha(f)+\nu_\alpha(g)=\nu_\alpha(fg)<\nu_\beta(fg)=\nu_\beta(f)+\nu_\beta(g)
\]
for every $\alpha,\beta\in A$ with $\alpha<\beta$. This is impossible if $f,g\in\mathcal C(\mathcal F)$. Hence $f\in\overline\Phi(\mathcal F)$ or $g\in\overline\Phi(\mathcal F)$.
\end{proof}

Let $d:=d(\mathcal F)$ be the smallest degree of a polynomial in $\overline{\Phi}(\mathcal F)$ and
\[
\Phi(\mathcal F):=\{q\in \overline{\Phi}(\mathcal F)\mid \deg(q)=d\}.
\]
In particular, $K[x]_d\subseteq\mathcal C(F)$. We can define the map
\[
\nu_{\mathcal F}:\mathcal C(\mathcal F)\lra \Gamma\mbox{ by }\nu_{\mathcal F}(f):=\nu_{\alpha_f}(f).
\]

\begin{Obs}\label{maksfinitoskesmoaf}
For a finite number of polynomials $f_1,\ldots,f_n\in \mathcal C(\mathcal F)$ we have
\[
\nu_{\mathcal F}(f_i)=\nu_\alpha(f_i)\mbox{ for every }i,1\leq i\leq n,\mbox{ for }\alpha=\max_{1\leq i\leq n}\{\alpha_{f_i}\}.
\]
If $f,g\in K[x]_d$, then $f,g,f+g\in K[x]_d$ and $fg\in \mathcal C(\mathcal F)$. Hence, there exists $\alpha$ such that
\[
\nu_{\mathcal F}(f)=\nu_\alpha(f), \nu_{\mathcal F}(g)=\nu_\alpha(g), \nu_{\mathcal F}(f+g)=\nu_\alpha(f+g)\mbox{ and }\nu_{\mathcal F}(fg)=\nu_\alpha(fg).
\]
Therefore,
\[
\nu_{\mathcal{F}}(f+g)\geq\min\{\nu_{\mathcal{F}}(f),\nu_{\mathcal{F}}(g)\}\mbox{ and }\nu_{\mathcal F}(fg)=\nu_{\mathcal F}(f)+\nu_{\mathcal F}(g).
\]
\end{Obs}

\begin{Def}
A monic polynomial $Q\in K[x]$ is said to be a \textbf{limit MacLane-Vaqui\'e key polynomial} for the continued family of iterated valuations $\mathcal F=\{(\nu_\alpha,Q_\alpha,\gamma_\alpha)\}_{\alpha\in A}$ if $Q$ has the following properties:
\begin{description}
\item[(LKP1)] $Q$ is $\mathcal F$-irreducible, i.e., for $f,g\in K[x]$, if $Q\mid_{\mathcal F}fg$, then $Q\mid_{\mathcal F} f$ or $Q\mid_{\mathcal F}g$.
\item[(LKP2)] $Q$ is $\mathcal F$-minimal, i.e., for $f\in K[x]$, if $Q\mid_{\mathcal F}f$, then $\deg(Q)\leq \deg(f)$.
\end{description} 
\end{Def}
The next result gives us a criterium to find limit key polynomials.
\begin{Prop}\label{Propositibnkelimivaquie}
Assume that $\mathcal F$ is a continued family of iterated valuations, with $\nu_\alpha\leq\nu_\beta$ if $\alpha<\beta$, and $Q\in \Phi(\mathcal F)$ is a monic polynomial. Then $Q$ is a limit MacLane-Vaqui\'e key polynomial for $\mathcal F$.
\end{Prop}
We will need the following lemma.
\begin{Lema}\label{Lemaondevqquieminti}
Assume that $\mathcal F$ is a continued family of iterated valuations with $\nu_\alpha\leq\nu_\beta$ if $\alpha<\beta$. If $Q\in\Phi(\mathcal F)$, then for $f\in K[x]$, we have
\[
Q\mid_{\mathcal F}f\Llr f\in\overline\Phi(\mathcal F).
\]
\end{Lema}
\begin{proof}
Take a polynomial $f$ with $Q\mid_{\mathcal F}f$ and suppose that $f\in \mathcal C(\mathcal F)$. Since $Q\mid_{\mathcal F}f$, there exists $\alpha_0\in A$ such that $Q\mid_{\nu_\alpha} f$ for every $\alpha\geq \alpha_0$. For $\alpha>\max\{\alpha_0,\alpha_f\}$, there exists $a\in K[x]$ such that
\[
\nu_\alpha(f-aQ)>\nu_\alpha(f)=\nu_\alpha(aQ).
\]
Take $\beta\in A$ with $\beta>\alpha$. Then
\[
\nu_\beta(f-aQ)\geq \nu_\alpha(f-aQ)>\nu_\alpha(f)=\nu_\beta(f)
\]
and consequently
\[
\nu_\beta(aQ)=\nu_\beta(f)=\nu_\alpha(f)=\nu_\alpha(aQ)
\]
contradicting $Q\in \Phi(\mathcal F)$.

For the converse, assume that $f\in \overline \Phi(\mathcal F)$. Write $f=aQ+r$ with $\deg(r)<\deg(Q)$. Since $\deg(r)<\deg(Q)$, we have $r\in\mathcal C(\mathcal F)$. We claim that, for every $\alpha\in A$ with $\alpha\geq\alpha_r$, we have
\begin{equation}\label{equaqbnaosutil}
\nu_\alpha(f-aQ)=\nu_\alpha(r)>\min\{\nu_\alpha(f),\nu_\alpha(aQ)\}
\end{equation}
and consequently $Q\mid_{\mathcal F} f$. Indeed, if $\nu_\alpha(r)\leq \min\{\nu_\alpha(f),\nu_\alpha(aQ)\}$, then for $\beta>\alpha$ we would have
\[
\nu_\beta(r)=\nu_\alpha(r)\leq \min\{\nu_\alpha(f),\nu_\alpha(aQ)\}<\min\{\nu_\beta(f),\nu_\beta(aQ)\}
\]
and this is a contradiction to fact that $\nu_\beta$ is a valuation.
\end{proof}

\begin{proof}[Proof of Proposition \ref{Propositibnkelimivaquie}]
Take $f\in K[x]$. If $Q\mid_{\mathcal F}f$, then by the previous lemma, $f\in\overline\Phi(\mathcal F)$. Since $Q\in\Phi(\mathcal F)$, we have $\deg(Q)\leq\deg(f)$. Hence, $Q$ satisfies \textbf{(LKP2)}.

Assume now that $Q\mid_{\mathcal F}fg$. Then, by Lemma \ref{Lemaondevqquieminti}, $fg\in\overline\Phi(\mathcal F)$. By Corollary \ref{corquepareceprimosnaksk}, this implies that $f\in\overline\Phi(\mathcal F)$ or $g\in\overline\Phi(\mathcal F)$ and again by Lemma \ref{Lemaondevqquieminti} we obtain that $Q\mid_{\mathcal F}f$ or $Q\mid_{\mathcal F}g$. Therefore, $Q$ is a limit MacLane-Vaqui\'e key polynomial for $\mathcal F$.
\end{proof}

\begin{Teo}\label{Theoremsobreexntvalaugvaeui}
Let $\mathcal F$ be a continued family of iterated valuations and $Q$ a limit key polynomial for $\nu$, with $d=\deg(Q)\leq d(\mathcal F)$. Take $\gamma$ in some extension of $\Gamma$ such that $\gamma>\nu_\alpha(Q)$ for every $\alpha\in A$. Define
\[
\overline{\nu}(f)=\min\{\nu_{\mathcal F}(f_i)+i\gamma\}
\]
where $f=f_0+f_1Q+\ldots+f_rQ^r$ is the $Q$-expansion of $f$. Then $\overline{\nu}$ is a valuation on $K[x]$.
\end{Teo}
We will need the following lemma (which is the equivalent of Lemma \ref{primeiraparteleloup} for an iterated family of valuations).

\begin{Lema}
For $f,g\in K[x]_d$, if $fg=aQ+r$ with $\deg(r)<\deg(Q)$, then
\begin{equation}\label{eqsobreeunemseimaisidjl}
\nu_{\mathcal F}(fg)=\nu_{\mathcal F}(r)<\nu_{\mathcal F}(a)+\gamma.
\end{equation}
\end{Lema}
\begin{proof}
Suppose, aiming for a contradiction, that \eqref{eqsobreeunemseimaisidjl} is not satisfied. Then there exists $\alpha_0\in A$ such that $\nu_\alpha(fg)=\nu_{\mathcal F}(fg)$, $\nu_\alpha(r)=\nu_{\mathcal F}(r)$, and
\begin{equation}\label{ainbdamaisdoidaeenex1}
\max\{\nu_{\mathcal F}(r),\nu_{\mathcal F}(fg)\}>\nu_\alpha(aQ)\mbox{ for every }\alpha\geq\alpha_0.
\end{equation}
Indeed, if
\[
\max\{\nu_{\mathcal F}(r),\nu_{\mathcal F}(fg)\}\geq \nu_{\mathcal F}(a)+\gamma,
\]
then we have \eqref{ainbdamaisdoidaeenex1} because 
\[
\nu_{\mathcal F}(a)=\nu_\alpha(a)\mbox{ and }\gamma>\nu_\alpha(Q)
\]
for every $\alpha\geq\alpha_a$. On the other hand, if $\nu_{\mathcal F}(r)\neq\nu_{\mathcal F}(fg)$, then for every $\alpha\geq \max\{\alpha_{fg},\alpha_r\}$ we have
\[
\nu_\alpha(r)=\nu_{\mathcal F}(r)\neq\nu_{\mathcal F}(fg)=\nu_\alpha(fg)
\]
and consequently
\[
\max\{\nu_{\mathcal F}(r),\nu_{\mathcal F}(fg)\}=\max\{\nu_\alpha(r),\nu_\alpha(fg)\}>\min\{\nu_\alpha(r),\nu_\alpha(fg)\}=\nu_\alpha(aQ).
\]

We can assume that $\alpha_0\geq \alpha_r,\alpha_{fg}$. If $\nu_{\mathcal F}(fg)=\nu_\alpha(fg)> \nu_\alpha(aQ)$, then
\[
\nu_\alpha(aQ)=\nu_\alpha(r)<\nu_\alpha(aQ+r)=\nu_\alpha(fg)
\]
for every $\alpha\geq\alpha_0$. Hence $Q\mid_{\mathcal F} r$. Analogously, if $\nu_{\mathcal F}(r)>\nu_\alpha(aQ)$ for every $\alpha\geq\alpha_0$, then $Q\mid_{\mathcal F} fg$ and since $Q$ satisfies \textbf{(LKP1)} we conclude that $Q\mid_{\mathcal F}f$ or $Q\mid_{\mathcal F} g$. In each case we obtain a contradiction, because $\max\{\deg(r),\deg(f),\deg(g)\}<\deg(Q)$ and $Q$ satisfies \textbf{(LKP2)}.
\end{proof}

\begin{proof}
Property \textbf{(V3)} follows by definition. By Remark \ref{maksfinitoskesmoaf}, the assumptions of Lemma \ref{Lemasobreextsvelidesitra} are satisfied for $\mu=\nu_{\mathcal F}$ and consequently \textbf{(V2)} is satisfied for $\overline\nu$.

Again by Remark \ref{maksfinitoskesmoaf} the condition \textbf{(i)} of Lemma \ref{lemasobreooutrocoiso} is satisfied for $\mu=\nu_{\mathcal F}$. Moreover, by the previous Lemma the condition \textbf{(ii)} of Lemma \ref{lemasobreooutrocoiso} is satisfied for $\mu=\nu_{\mathcal F}$, we obtain that $\mu'=\overline{\nu}$ satisfies \textbf{(V1)}. Theorefore, $\overline\nu$ is a valuation.
\end{proof}

Before ending this section we will discuss when the condition $\nu_\alpha\leq\nu_\beta$ for $\alpha<\beta$ is satisfied. We start with the following proposition.
\begin{Prop}\label{Lemainteressantemasfalo}
Let $\nu$ be a valuation on $K[x]$ and let
\[
\nu_1=[\nu;\nu_1(Q_1)=\gamma_1]\mbox{ and }\nu_2=[\nu_1;\nu_2(Q_2)=\gamma_2]
\]
be augmented valuations and assume that $Q_1\nsim_{\nu_1}Q_2$. If $\deg(Q_1)=\deg(Q_2)$, then $\gamma_2>\gamma_1$ and
\[
\nu_2(Q_1)=\nu_1(Q_2)=\gamma_1=\nu(Q_2-Q_1).
\]
\end{Prop}
\begin{proof}
Since $Q_1$ and $Q_2$ are monic polynomials of the same degree, we have that $\deg(h)<\deg(Q_1)=\deg(Q_2)$ where $h=Q_2-Q_1$. In particular,
\[
\nu(h)=\nu_1(h)=\nu_2(h).
\]
Since $Q_2$ is a MacLane-Vaqui\'e key polynomial for $\nu_1$, by Lemma \ref{primeiraparteleloup} \textbf{(i)} (for $f=Q_1$, $aQ=Q_2$ and $r=h$) we have that
\begin{equation}\label{relapodlequanquerosusa}
\nu(h)=\nu_1(h)\geq \nu_1(Q_1)=\gamma_1\mbox{ and }\nu_1(Q_2)\geq \nu_1(Q_1).
\end{equation}
By the definition of $\nu_1$ and by \eqref{relapodlequanquerosusa} we have
\[
\nu_1(Q_2)=\min\{\gamma_1,\nu(h)\}=\gamma_1.
\]
Hence, $\gamma_2>\nu_1(Q_2)=\gamma_1$.

Since $Q_1\nsim_{\nu_1}Q_2$, we have
\[
\nu(h)=\nu_1(h)=\nu(Q_2-Q_1)\leq \nu_1(Q_1)=\gamma_1.
\]
This and \eqref{relapodlequanquerosusa} imply that $\nu_1(h)=\nu(h)=\gamma_1$. Hence,
\[
\nu_2(Q_1)=\nu_2(Q_2-h)=\min\{\gamma_2,\nu_1(h)\}=\gamma_1.
\]
\end{proof}
\begin{Cor}
On the situation of Lemma \ref{Lemainteressantemasfalo} we have that
\[
\nu_2=[\nu;\nu_2(Q_2)=\gamma_2].
\]
\end{Cor}
\begin{proof}
For a polynomial $f\in K[x]$ let
\[
f=f_0+f_1Q_2+\ldots+f_rQ_2^r
\]
be the $Q_2$-expansion of $f$. If $Q_2$ is a key polynomial for $\nu$, then
\[
\nu_2(f)=\min\{\nu_1(f_i)+i\gamma_2\}=\min\{\nu(f_i)+i\gamma_2\}.
\]
Hence, we only need to show that $Q_2$ is a key polynomial for $\nu$.

By Proposition \ref{Lemainteressantemasfalo}
\[
\nu(Q_2-Q_1)=\gamma_1>\nu(Q_1)
\]
and hence $Q_2\sim_\nu Q_1$. This, and the fact that $Q_1$ is a key polynomial for $\nu$ imply that $Q_2$ is a key polynomial for $\nu$.
\end{proof}

\begin{Que}\label{quaesobregoblso}
Is the converse of the previous corollary true? More precisely, assume that
\[
\nu_1=[\nu;\nu_1(Q_1)=\gamma_1]\mbox{ and }\nu_2=[\nu;\nu_2(Q_2)=\gamma_2],
\]
with $\deg(Q_2)=\deg(Q_1)$ and $\gamma_1<\gamma_2$. Is it true that
\begin{equation}\label{reqajhsijdmasovaquidoemanf}
\nu_2=[\nu_1;\nu_2(Q_2)=\gamma_2]?
\end{equation}
\end{Que}

If \eqref{reqajhsijdmasovaquidoemanf} is satisfied, then $\nu_1\leq \nu_2$. Observe that in the construction of this section, we used the property $\nu_1\leq\nu_2$ rather than \eqref{reqajhsijdmasovaquidoemanf}. The next lemma gives an easy criterium of when $\nu_1\leq\nu_2$.

\begin{Lema}
Let $\nu$ be a valuation on $K[x]$ and
\[
\nu_1=[\nu;\nu_1(Q_1)=\gamma_1]\mbox{ and }\nu_1=[\nu;\nu_2(Q_2)=\gamma_2]
\]
be two augmented valuations. Assume that $\deg(Q_1)=\deg(Q_2)$ and that $\gamma_1<\gamma_2$. Then, 
\[
\gamma_1=\nu_1(Q_1)\leq \nu_2(Q_1)\Llr \nu_1\leq \nu_2.
\]
\end{Lema}
\begin{proof}
The implication $``\Longleftarrow"$ is trivial. For the converse, assume that $\gamma_1=\nu_1(Q_1)\leq \nu_2(Q_1)$. For a given $f\in K[x]$, let
\[
f=f_0+f_1Q_1+\ldots+f_rQ_1^r
\]
be the $Q_1$-expansion of $f$. Then
\[
\nu_2(f)\geq\min_{0\leq i\leq r}\{\nu_2(f_iQ_1^i)\}=\min_{0\leq i\leq r}\{\nu(f_i)+i\nu_2(Q_1)\}\geq\min_{0\leq i\leq r}\{\nu(f_i)+i\gamma_1\}=\nu_1(f).
\]
\end{proof}

\begin{Obs}
We observe that in the situation above $\nu(Q_2)=\nu(Q_1)$. Indeed, if $\nu(Q_1)<\nu(Q_2)$, then
\[
\nu(Q_2-(Q_2-Q_1))=\nu(Q_2)>\nu(Q_1)
\]
and consequently $Q_1\sim_\nu Q_2-Q_1$ which is a contradiction to \textbf{(KP2)}. The case $\nu(Q_2)<\nu(Q_1)$ is analogous.
\end{Obs}

The next example shows that in our situation, $\nu_1\leq\nu_2$ (consequently \eqref{reqajhsijdmasovaquidoemanf}) is not necessarily true.
\begin{Exa}
Let $\nu_t$ be the $t$-adic valuation on $K=k(t)$ and extend it to $K[x]$ by defining
\[
\nu(a_0+a_1x+\ldots+a_rx^r)=\min_{1\leq i\leq r}\{\nu_t(a_i)+i\}\mbox{, i.e., }\nu=[\nu_t;\nu(x)=1].
\]
The polynomials $Q_1=x-t$ and $Q_2=x-t-t^2$ are MacLane-Vaqui\'e key polynomials for $\nu$. Define
\[
\nu_1=[\nu;\nu_1(Q_1)=3]\mbox{ and }\nu_2=[\nu;\nu_2(Q_2)=4].
\]
Then $\nu_2(Q_1)=\min\{4,\nu(Q_2-Q_1)\}=2<3=\nu_1(Q_1)$.
\end{Exa}

\section{Key polynomials vs MacLane-Vaqui\'e key polynomials}\label{comparison}
The main goal of this section is to relate MacLane-Vaqui\'e key polynomials with key polynomials. We start with the following result.
\begin{Teo}\label{theomremnaequilvamahs}
Let $\nu$ be a valuation on $K[x]$, and take a key polynomial $Q$ and $Q'\in \Psi(Q)$. Then $Q$ and $Q'$ are MacLane-Vaqui\'e key polynomials for $\nu_Q$. Moreover, $\nu_{Q'}=[\nu_Q;\nu_{Q'}(Q')=\nu(Q')]$.
\end{Teo}
\begin{proof}
Corollaries \ref{Corlsobremaclaneparatruncame} and \ref{Corlsobremaclaneparatruncame1} give us that $Q$ is a MacLane-Vaqui\'e key polynomial for $\nu_Q$. Also, Lemma \ref{lemaqueajudanapropsdoirredl} and Corollary \ref{Corlsobremaclaneparatruncame3} imply that $Q'$ is MacLane-Vaqui\'e key polynomial for $\nu_Q$. For the last statement, we observe that $\nu_Q(f)=\nu(f)$ for every $f\in K[x]$ with $\deg(f)<\deg(Q')$ because of the minimality of the degre of $Q'$. Hence
\[
\nu_{Q'}(f_0+\ldots+f_rQ'^r)=\min_{0\leq i\leq r}\{\nu(f_iQ'^i)\}=\min_{0\leq i\leq r}\{\nu_Q(f_i)+i\nu(Q')\}.
\]
\end{proof}

\begin{Teo}\label{maintheomrlimitvsvauqiemabca}
Assume that $\nu$ is a valuation on $K[x]$ and that $Q$ is a limit key polynomial for $\nu$. Then the family
\[
\mathcal F=\{(\nu_{Q'}, Q',\nu(Q'))\}_{Q'\in\Psi(Q_-)}
\]
ordered by $\epsilon$, is a continued family of augmented valuations on $K[x]$ and $Q$ is a limit MacLane-Vaqui\'e key polynomial for $\nu$. Moreover,
\[
\nu_Q=[\nu_{\mathcal F};\nu_Q(Q)=\nu(Q)].
\]
\end{Teo}
In order to prove Theorem \ref{maintheomrlimitvsvauqiemabca}, we will need the following Lemma.
\begin{Lema}\label{lemasobrearelatimlimvaqueicomlimite}
Assume that $\nu$ is a valuation on $K[x]$ and that $Q$ is a key polynomial for $\nu$. For $Q',Q''\in \Psi(Q)$, if $\epsilon(Q')<\epsilon(Q'')$, then
\[
\nu_{Q''}=[\nu_{Q'};\nu_{Q''}(Q'')=\nu(Q'')].
\]
Moreover, $Q'\nsim_{Q'}Q''$
\end{Lema}
\begin{proof}
Since $\epsilon(Q')<\epsilon(Q'')$, by Proposition \ref{Propcompkeypol} \textbf{(ii)}, we have that $\nu_{Q'}(Q'')<\nu(Q'')$. Since $\deg(Q')=\deg(Q'')=\alpha(Q)$ we conclude that $Q''\in \Psi(Q')$. Hence, Theorem \ref{theomremnaequilvamahs} gives us that $Q''$ is a MacLane-Vaqui\'e key polynomial for $\nu_{Q''}$ and $\nu_{Q''}=[\nu_{Q'};\nu_{Q''}(Q'')=\nu(Q'')]$.

Since $\deg(Q')=\deg(Q'')$ and $\epsilon(Q')<\epsilon(Q'')$, Proposition \ref{Propcompkeypol} \textbf{(iii)} gives us that $\nu(Q')<\nu(Q'')$ and hence
\[
\nu_{Q'}(Q'-Q'')=\nu(Q'-Q'')=\nu(Q')=\nu_{Q'}(Q').
\]
Consequently, $Q'\nsim_{Q'}Q''$.
\end{proof}
\begin{proof}[Proof of Theorem \ref{maintheomrlimitvsvauqiemabca}]
We will start by proving that $\mathcal F$ is an iterated family of augmented valuations. Take $Q'\in \Psi(Q_-)$ and assume that it is not the smallest element of $\Psi(Q_-)$. If $Q'$ admits predecessor, we set $Q'_-$ to be that predecessor. If not, set $Q'_-$ to be any element in $\Psi(Q_-)$ such that $\epsilon(Q'_-)<\epsilon(Q')$. We will show that $Q'_-$ satisfy the conditions of the definition of iterated family of valuations (for $\alpha=Q'$ and $\alpha_-=Q'_-$).

Observe that by Lemma \ref{lemasobrearelatimlimvaqueicomlimite} we have
\[
\nu_{Q'}=[\nu_{Q'_-};\nu_{Q'}(Q')=\nu(Q')].
\]
Moreover, since $\deg(Q')=\alpha(Q_-)$ for every $Q'\in\Psi(Q_-)$, the conditions on the degrees are automatically satisified.

Assume that we are in case \textbf{(i)}, i.e., that $Q'$ admits a predecessor. By definition $Q'_-$ is that predecessor and by Lemma \ref{lemasobrearelatimlimvaqueicomlimite} we have $Q'_-\nsim_{Q'_-}Q'$. Hence, the condition of Definition \ref{defoffamiliaugmentedintereda} is satisified.

Assume that we are in case \textbf{(ii)}, i.e., that $Q'$ does not admit a predecessor, and take $Q''\in\Psi(Q_-)$ with
\[
\epsilon(Q'_-)<\epsilon(Q'')<\epsilon(Q').
\]
By Lemma \ref{lemasobrearelatimlimvaqueicomlimite}, we have that
\[
\nu_{Q''}=[\nu_{Q'_-};\nu_{Q''}(Q'')=\nu(Q'')]\mbox{ and }\nu_{Q'}=[\nu_{Q''};\nu_{Q'}(Q')=\nu(Q')].
\]
Therefore, $\mathcal F$ is a family of iterated valuations.

The fact that it is continued follows from the fact that for every $Q'\in\Psi(Q_-)$ we have (by Theorem \ref{theomremnaequilvamahs}) that
\[
\nu_{Q''}=[\nu_{Q_-};\nu_{Q''}(Q'')=\nu(Q'')].
\]

It remains to prove that $Q$ is a limit MacLane-Vaqui\'e key polynomial for $\mathcal F$ and that $\nu_Q=[\nu_{\mathcal F};\nu_Q(Q)=\nu(Q)]$. Observe that for $Q',Q''\in\Psi(Q_-)$, with $\epsilon(Q')<\epsilon(Q'')$, Lemma \ref{lemasobrearelatimlimvaqueicomlimite}, we have that $\nu_{Q'}\leq \nu_{Q''}$. Hence, if we prove that $Q\in\Phi(\mathcal F)$ the result will follow from Proposition \ref{Propositibnkelimivaquie}.

By \textbf{(K3)} we have $\nu_{Q'}(Q)<\nu(Q)$ for every $Q'\in \Psi(Q_-)$. In particular, $\{\nu_{Q'}(Q)\}_{Q'\in\Psi(Q_-)}$ is increasing (see Corollary \ref{corsobreequaldepsoidkey}). Hence, $Q\in\overline\Phi(\mathcal F)$. Assume that $f\in \overline\Phi(\mathcal F)$. Then, $\{\nu_{Q'}(f)\}_{Q'\in\Psi(Q_-)}$ is increasing and consequently the condition \textbf{(K3)} is satisfied for $f$. By \textbf{(K4)} we conclude that $\deg(Q)\leq \deg(f)$. Consequently, $Q\in \Phi(\mathcal F)$.

For any polynomial $f\in K[x]$, let $f=f_0+f_1Q+\ldots+f_rQ^r$ be the $Q$-expansion of $f$. Then there exists $Q'\in \Psi(Q_-)$ such that $\nu_{Q''}(f_i)=\nu(f_i)$ for every $i$, $0\leq i\leq r$ and $Q''\in\Psi(Q_-)$ with $\epsilon(Q')\leq \epsilon(Q'')$. In particular, $\nu_{\mathcal F}(f_i)=\nu(f_i)$. Hence,
\[
\nu_Q(f)=\min\{\nu(f_iQ^i)\}=\min\{\nu_{\mathcal F}(f_i)+\nu(Q)\}
\]
and therefore, $\nu_Q=[\nu_{\mathcal F};\nu_Q(Q)=\nu(Q)]$.
\end{proof}
The next corollary is the main result of \cite{Mac_1} and \cite{Vaq_1}.
\begin{Cor}
For every valuation $\nu$ on $K[x]$, there exists a family of augmented iterated valuations $\mathcal F=\{(\nu_\alpha,Q_\alpha,\gamma_\alpha)\}_{\alpha\in A}$ such that for every $f\in K[x]$, there exists $\alpha_0\in A$ for which $\nu(f)=\nu_\alpha(f)$ for every $\alpha\in A$, $\alpha\geq \alpha_0$.
\end{Cor}
\begin{proof}
By Theorem \ref{Theoremexistencecompleteseqkpol}, $\nu$ admits a sequence (ordered by $\epsilon$) $\textbf{Q}$ of key polynomials. Since $\textbf{Q}$ can be chosen to be well-odered, for every $Q\in \textbf{Q}$ we set $Q^+$ to be the next element (i.e., the element with smallest $\epsilon$ in $\{Q'\in \textbf{Q}\mid \epsilon(Q)<\epsilon(Q')\}$). By Theorems \ref{theomremnaequilvamahs} and \ref{maintheomrlimitvsvauqiemabca}, the family $\mathcal F=\{(\nu_Q,Q^+,\nu(Q^+))\}_{Q\in\textbf{Q}}$ is a family of augmented iterated valuations and this concludes our proof.
\end{proof}

\noindent{\footnotesize JOSNEI NOVACOSKI\\
Departamento de Matem\'atica--UFSCar\\
Rodovia Washington Lu\'is, 235\\
13565-905, S\~ao Carlos - SP, Brazil.\\
Email: {\tt josnei@dm.ufscar.br} \\\\

\end{document}